\documentclass[a4paper, 11pt]{amsart}
\input xy
\xyoption{all}
\swapnumbers
\usepackage[top=4cm, bottom=3.5cm, left=2cm, right=2cm,marginparwidth=1.8cm, marginparsep=0.1cm]{geometry}
\usepackage[OT2, T1]{fontenc}
\usepackage[english]{babel}
\usepackage{amssymb}
\usepackage{amsthm}
\usepackage{enumerate}
\usepackage[active]{srcltx}
\usepackage{amsmath}
\usepackage{nicefrac}
\usepackage{tikz-cd} 
\usepackage[colorlinks=true]{hyperref}
\usepackage{lipsum}
 \usepackage{xfrac}
\usepackage{faktor}
\usepackage{cite}
\usepackage[inline]{enumitem}

\providecommand{\bysame}{\leavevmode\hbox to3em{\hrulefill}\thinspace}
\providecommand{\MR}{\relax\ifhmode\unskip\space\fi MR }

\providecommand{\href}[2]{#2}

\newtheorem{proposition}{Proposition}[section]
\newtheorem{theorem}[proposition]{Theorem}
\newtheorem{lemma}[proposition]{Lemma}
\newtheorem{corollary}[proposition]{Corollary}

\newtheorem{claim}[proposition]{Claim}

\newtheorem{definition}[proposition]{Definition}

\theoremstyle{remark}
\newtheorem{remark}[proposition]{Remark}

\newtheorem{examples}[proposition]{Examples}

\newtheorem*{remark*}{Remark}
\newtheorem*{remarks*}{Remarks}

\numberwithin{equation}{section}

\newcommand{\CH}{\mathcal{H}}

\newcommand{\R}{\mathbb{R}}

\newcommand{\defeq}{:=}

\newcommand{\wrwr}{\operatorname{\wr\wr}}

\begin{document}
\title[Metric approximations of unrestricted wreath products]{Metric approximations of unrestricted wreath products when the acting group is amenable}

\author[J. Brude, R. Sasyk]{Javier Brude $^{1,2}$ \MakeLowercase{and} Rom\'an Sasyk $^{1,3}$}

\address{$^{1}$Instituto Argentino de Matem\'atica Alberto P. Calder\'on-CONICET,
Saavedra 15, Piso 3 (1083), Buenos Aires, Argentina.}

\address{$^{2}$Departamento de Matem\'atica, Facultad de Ciencias Exactas, Universidad Nacional de la Plata, Argentina.}

\address{$^{3}$Departamento de Matem\'atica, Facultad de Ciencias Exactas y Naturales, Universidad de Buenos Aires, Argentina.}

\email{\textcolor[rgb]{0.00,0.00,0.84}{jbrude@mate.unlp.edu.ar}}
\email{\textcolor[rgb]{0.00,0.00,0.84}{rsasyk@dm.uba.ar}}


\date{}

\begin{abstract}	
We give a simple and unified proof showing that the unrestricted wreath product of a weakly sofic, sofic, linear sofic, or hyperlinear group by an amenable group is weakly sofic, sofic, linear sofic, or hyperlinear, respectively. By means of the Kaloujnine-Krasner theorem, this implies that  group extensions with amenable quotients preserve the four aforementioned metric approximation properties. We also discuss the case of co-amenable groups.
\end{abstract}

\maketitle
\section{introduction}
Given two  groups $G$ and $H$, their {\it unrestricted wreath product} $G\wrwr H$ is, by definition, the semidirect product 
$\left(\prod_H G\right )\rtimes_\theta H$,
where $H$ acts on the direct product $\prod_H G$ by shifting coordinates as follows: $\theta_{h}\big(\left (x_{\tilde h}\right)_{\tilde h\in H}\big):=\left (x_{\tilde h}\right)_{h\tilde h\in H}$, for $\left (x_{\tilde h}\right)_{\tilde h\in H}\in \prod_H G$. 
 The purpose of this article is to provide a simple and unified proof of the following statement.
\begin{theorem}\label{main teorema}
Let $H$ be an amenable group and let $G$ be a group. 
\begin{enumerate}
\item If $G$ is weakly sofic, $G\wrwr H$ is weakly sofic. \label{caso weak sofic}
\item If $G$ is sofic, $G\wrwr H$ is sofic. \label{caso  sofic}
\item If $G$ is linear sofic, $G\wrwr H$ is linear sofic. \label{caso linear  sofic}
\item If $G$ is hyperlinear, $G\wrwr H$ is hyperlinear. \label{caso hyperlineal}
\end{enumerate}
\end{theorem}

By means of the Kaloujnine-Krasner theorem, \cite{MR49891}, and by keeping in mind that
these four metric approximation properties pass to subgroups and are preserved under taking direct products, 
it is easy to see that Theorem \ref{main teorema} and the next result are equivalent.

\begin{corollary}[Extension Theorem]\label{coro on extensions} Let $G$ be a  group with a normal subgroup $N$ such that the quotient $G/N$ is amenable. 
If $N$ is weakly sofic, sofic, linear sofic, or hyperlinear;
then $G$ is  weakly sofic,  sofic, linear sofic, or hyperlinear, respectively. 
\end{corollary}
The sofic and hyperlinear cases of  Theorem \ref{main teorema} have been proved by Arzhantseva, Berlai, Finn-Sell, and Glebsky in \cite{MR3934790}, where they also gave the application of the Kaloujnine-Krasner theorem mentioned above. Some of the extension results are older. Indeed, the sofic one is due to Elek and Szabo in \cite{ELE03} and the linear sofic one is due to Arzhantseva and P\u{a}unescu in \cite{MR3592512}. More recently, in  \cite{MR3632893}, Holt and Rees showed that certain metric approximations on groups, including weakly sofic, are preserved under taking extensions with amenable quotients.  In a slightly different direction, in \cite{HAY18}, Hayes and Sale proved that the {\it restricted wreath product} of a group $G$ having one of the metric approximation properties listed in Theorem \ref{main teorema} by an acting sofic group, preserves the approximation property of $G$.

 In \cite[$\S$4.3]{MR3934790} the authors explained why their techniques could not deal with the weakly sofic case of Theorem  \ref{main teorema}. 
 The motivation of the present article was to see whether ideas we used in \cite[$\S$5]{BS2019}, some of which can be traced back to \cite{HAY18,MR3632893},
would serve to give a direct proof of this fact, without requiring the result on extensions of \cite{MR3632893}. 
Here we achieve this in a self-contained manner that also allows us to deal with the four cases of Theorem \ref{main teorema} in a unified way. 
The reasons that make our proof simple are, on the one hand, the systematic use 
 of the notion of abstract metric approximations, considered first by
Arzhantseva in 2008 (see  \cite{MR3644759})
and on the other hand, that we deal directly with metric approximations of $\prod_H G$ rather than building them locally from metric approximations of $G$. This is done in Section \ref{S3}.
The last ingredient in our proof, which might be useful in other contexts, appears in Section \ref{S4}. There we provide explicit group monomorphisms from certain auxiliary metric groups constructed in Section \ref{S3}, to
 metric groups belonging to the classes defining the metric properties of Theorem \ref{main teorema}, in a way that the metric is distorted in a controlled manner. 
 
 The Extension Theorem \ref{coro on extensions} admits a more general form where $N$ is replaced by a (not necessarily normal) subgroup $H$ and the hypothesis of the amenability of $G/N$ are replaced by the assumption that the left regular action of $G$ on $G/H$ is amenable. In the literature, this is  sometimes referred as $H$ being {\it co-amenable} in $G$, (see, for instance, \cite{MR1999183}), and this is how we will call it in this article.
 In Section \ref{S5} we discuss why even the more general version of the  Kaloujnine-Krasner theorem \cite{MR49892}, might not be suitable
to prove this variant for co-amenable subgroups. Nonetheless,  we provide a unified proof of this more general form of the Extension Theorem by adapting the work done with permutational wreath products in Section \ref{S3} and in a manner which allows us to use the results of Section \ref{S4}.
\section{Metric approximation in groups}
In this section we give a very brief account of some basic notions of metric approximation in groups needed in this article.
 Given a group $G$, a map $\delta:G\setminus\{1\}\to (0,\infty)$ is called a {\it weight function} for $G$.

\begin{definition} Let $G$ be a group with a weight function $\delta$ and let $K$ be a group with a bi-invariant metric $d$.
 Given $F\subseteq G$, $\varepsilon>0$, and a map $\phi:G\to K$ such that $\phi (1)=1$, we say that
\begin{enumerate} 
\item $\phi$ is $(F,\varepsilon,d)$-multiplicative  if
$d(\phi(g)\phi(g'),\phi(gg')) < \varepsilon$ for all  $g,g'\in F$;
\item $\phi$ is $(F,\delta,d)$-injective if 
 $d(\phi(g),1) \geq  \delta(g)$ for all  $g\in F\setminus\{1\}$;
 \item $\phi$ is $(F,\varepsilon,d)$-free for a function $\rho(\varepsilon)$ 
  if $d(\phi(g),1) \geq  \rho(\varepsilon)$, for all  $g\in F\setminus\{1\}$.
\end {enumerate} 
\end{definition}
\begin{definition}[{\cite[Definition 1.6]{MR2900231}}] Let $\mathcal C$ be a class of groups with bi-invariant metrics.
 A group $G$ is $\mathcal C$-approximable or has the  $\mathcal C$-approximation property, if it has a weight function $\delta$ such that, for each  finite set $F\subseteq G$ and each 
 $\varepsilon>0$ there exist $(K,d)\in \mathcal C$ and an $(F,\varepsilon,d)$-multiplicative function $\phi:G\to K$ which is also $(F,\delta,d)$-injective. 
\end{definition}

A {\it countable} group is $\mathcal C$-approximable if and only if it is embeddable  in a metric ultraproduct of groups in $\mathcal C$. Hence, the importance of this notion comes from its relation to major open challenges in mathematics like the Connes' embedding problem for group von Neumann algebras and the Gottschalk's surjunctivity conjecture, (see \cite [Proposition 1.8]{MR2900231} and references therein).

\begin{examples}\label{ejemplos} We list the four metric approximation properties studied in this article.
\begin{enumerate}
\item \label {w sofico}A group is {\it weakly sofic}, \cite{MR2455513}, if it is $\mathcal C$-approximable when $\mathcal C$ is the class of finite groups with bi-invariant metrics. 
\item \label{sofico}A group is {\it sofic}, \cite{MR1803462,ELE03}, if it is $\mathcal C$-approximable when $\mathcal C$ is the class of finite symmetric groups endowed with the normalized Hamming distance 
$$d_{\rm{Hamm}}(\sigma,\tau):= \frac{1}{| A|}|\{a\in  A: \sigma(a)\neq \tau(a)\}| \text{ , for } \sigma, \tau\in Sym( A).$$
\item \label{lineal sofico}A group is  {\it linear sofic}, \cite{MR3592512}, if it is $\mathcal C$-approximable when $\mathcal C$ is the class of invertible matrices of finite rank on a field $K$ 
endowed with the rank distance given by 
$$d_{{\rm rk}}(A,B):=\frac{1}{n}{\rm rank }(A-B)  \text{ , for } A,B \in GL_{n}(K).\\$$
\item \label{hiperlin} A group is  {\it hyperlinear}, \cite{MR2436761,MR2072092}, if it is $\mathcal C$-approximable when $\mathcal C$ is the class of unitary matrices on a finite dimensional Hilbert space endowed with the Hilbert-Schmidt distance, $d_{{\rm HS}}$, that is induced by the Hilbert-Schmidt norm
$$\Vert A \Vert_2:=\sqrt{\frac{1}{n}\sum_{i,j=1}^{n}
|\langle Av_i,v_j\rangle | ^2 } \text{ , for } A \in \mathcal B(\mathcal H), \text{ where }\{v_1,\ldots,v_n\} \text{ is  an ONB of }\mathcal H.$$
\end{enumerate}
\end{examples}

  In these examples, $(F,\delta,d)$-injectivity can be replaced by a more manageable condition. Rather than requiring a weight function $\delta$ that depends on $G$, this gets replaced in each case as follows:
\begin{enumerate}
\item   {\it weakly sofic:} $\phi$ is $(F,\varepsilon,d)$-free for a constant function $\rho(\varepsilon):=\alpha \in \R_{>0}$, (this follows easily from the definition).
By scaling the metric $d$, it can be assumed that $\alpha=1$ and ${\rm diam}(K)=1$;
\item   {\it sofic:} $\phi$ is $(F,\varepsilon,d_{{\rm Hamm}})$-free for the function $\rho(\varepsilon):=1-\varepsilon$, 
(see \cite[Proposition 4.4]{MR2089244});
\item {\it linear sofic:} $\phi$ is $(F,\varepsilon,d_{{\rm rk}})$-free for the function $\rho(\varepsilon):=1/4-\varepsilon$, (see \cite[Proposition 5.13]{MR3592512});
\item  {\it  hyperlinear:}\label{trace preserving} $\phi$ is $(F,\varepsilon,d_{{\rm HS}})$-trace preserving, namely  $|tr(\phi(g))| \leq \varepsilon$ for all  $g\in F\setminus\{1\}$,
where, for $ A \in \mathcal B(\mathcal H)$ and $\{v_1,\ldots,v_n\}$  an ONB of $\mathcal H$, 
$tr(A):=\frac{1}{n}\sum_{i=1}^{n} \langle Av_i,v_i\rangle$, (see \cite[Proposition 2.5]{MR2436761}, and Remark \ref{about radulescu theorem} in this article).
\end{enumerate}
One of the advantages of having these alternative definitions is that each one is a uniform condition, in the sense that, unlike a weight function, they do not depend on the particular group $G$ that is $\mathcal C$-approximable. For instance,  it is easy to show that the four metric approximations given in Examples \ref{ejemplos} are preserved under taking finite direct products. But then, freeness implies that they are also preserved under taking direct products. 
\section{Approximations using permutational wreath products}\label{S3}
Given a group $K$ with a bi-invariant metric $d$ for which ${\rm diam}(K)\leq1$ and a finite set $B$, 
consider the {\it permutational wreath product} $K\wr_B Sym(B)$.
On a few occasions we will denote with a dot ``$\cdot$'' the permutational action of $Sym(B)$ on $\bigoplus_{B} K$.  
In \cite[$\S$5]{MR3632893} (see also \cite[Proposition 2.9]{HAY18}) the authors introduced the following function 
\begin{equation}\label{metric in permut. wreath product}\tilde d(((x_b)_{b\in B},\tau),((y_b)_{b\in B},\rho) ):=
d_{\text{Hamm}}(\tau,\rho)+\frac{1}{|B|}\sum_{\stackrel{b\in B}{\tau(b)=\rho(b)}} d(x_{\tau(b)},y_{\tau(b)})
\end{equation}
and showed that it is a bi-invariant metric in $K\wr_B Sym(B)$, and with this metric ${\rm diam}(K\wr_B Sym(B))=1$.

\begin{proposition}\label{weak sofic permanence}
Let $H$ be an amenable group and
let $G$ be a  group with the property that for every $\varepsilon'>0$ and for every finite set $F_{\,\prod_{H}G}\subseteq \prod_H G$,  
there exist a group  $(K,d)\in \mathcal C$ of diameter bounded by $1$ and a function 
$\varphi: \prod_H G\to K$ with $\varphi(1)=1$ that is $(F_{\,\prod_{H}G},\varepsilon', d)$-multiplicative 
and $(F_{\,\prod_{H}G},\varepsilon', d)$-free for certain function $\rho(\varepsilon')$, independent of $K$.

Then, given $\varepsilon>0$ and a finite set $F\subseteq G\wrwr H$, there exist a finite set $B\subseteq H$, 
a group $(K,d)\in \mathcal C$ and a function 
$\Phi: G\wrwr H\to K\wr_B Sym( B)$ that is $(F,\varepsilon,\tilde d)$-multiplicative 
and $(F,\varepsilon, \tilde d)$-free
for the function $\tilde\rho(\varepsilon):=\min(1-\varepsilon/3, \rho(\varepsilon/3))$.
\end{proposition}

\begin{remark} 
\begin{enumerate}
\item The hypothesis on $G$ in Proposition \ref{weak sofic permanence} is satisfied if $G$ is
$\mathcal C$-approximable, when $\mathcal C$ is any of the families in  Examples \ref{ejemplos},
 (see $\S$\ref{subsection hiperlinear} for a discussion on the hyperlinear case).
\item
The most natural hypothesis on $G$ would have been that $\prod_H G$ is $\mathcal C$-approximable. 
However, with this condition, given the weight function $\delta$ for $\prod_H G$, our proof would only produce something like a weight 
function for $G\wrwr H$ that depends on $B$ and $\varepsilon$, and this is not enough to conclude  $(F,\delta,\tilde d)$-injectivity  for $G\wrwr H$. 
\item 
The obstruction just mentioned gives a partial indication of why our proof can not be adapted to show, for instance,  that the class of MF groups, 
(namely groups that are 
$\mathcal C$-approximable when $\mathcal C$ is the class of unitary matrices endowed with the operator norm, see \cite[Definition 2.7]{MR3049883} and \cite[$\S$4]{Thom-ICM}) is stable under taking extensions with amenable quotients.
\item It was drawn to our attention that Proposition \ref{weak sofic permanence} is similar to \cite[Theorem 5.1]{MR3632893}. The main differences in the statements are that we deal with the concrete group $G\wrwr H$ rather than with extensions with amenable quotients, (and this will simplify the proof), that we use the notion of {\it freeness} rather than the in principle more restrictive notion of {\it discrete $\mathcal C$-approximation} introduced in \cite{MR3632893}, and that we do not require stability of the class $\mathcal C$ under permutational wreath products. These last two differences indicates why 
\cite[Theorem 5.1]{MR3632893} serves to prove only the weakly sofic case of Theorem \ref{main teorema}.
Finally, the statement and the proof of Proposition \ref{weak sofic permanence} will easily be adapted to get a stronger conclusion in the hyperlinear case, (see Remark  \ref{about radulescu theorem} and the proof Theorem \ref{main teorema}\eqref{caso hyperlineal}).
\end{enumerate}
\end{remark}
\begin{proof}[Proof of Proposition \ref{weak sofic permanence}]
Call $\text{proj}_{\,\prod_{H}G}$ and $\text{proj}_{H}$ the projection maps from $G\wrwr H$ to $\prod_{H}G$ and $H$, respectively.
Let $F \subseteq G\wrwr H$ be a finite subset and let $\varepsilon >0$. Define $F_0\defeq F \cup \{1\}\cup F^{-1}$ and let 
$F_H:=  {\rm proj}_H(F_0)$.
Since $H$ is amenable, there exists a finite subset $B\subseteq H$ such that
\begin{align}\label{folner}
    \frac{|hB \bigtriangleup B|}{|B|}\leq \varepsilon/6\,,\,\,\,\,  \text{ for all } h \in F_H^2:=F_H . F_H
\end{align}
Let $\sigma: H \to Sym(B)$ be defined as
\begin{equation}\label{def of sigma}
       \sigma(h)b := \begin{cases} hb &\text{ if } hb \in B 
        \\
        \gamma_h {(hb)} &\text{ if not, where }  
       \gamma_h : hB \setminus B \to B \setminus hB \text{ is a fixed bijection}. 
         \end{cases} 
\end{equation} 
It is easy to see that 
$\sigma$ is $(F_H, \varepsilon/3,d_{{\rm Hamm}})$-multiplicative and $(F_H, \varepsilon/3,d_{{\rm Hamm}})$-free 
for the function $\rho'(\varepsilon/3)=1-\varepsilon/3$.

Let $\theta$ denote the shift action of $H$ on $\prod_H G$. In what follows it will be handy to keep in mind that since $\theta$ is an action, 
 $\theta_h\theta_{h'}=\theta_{hh'}$ for all $h,h'\in H$.
 
  By the hypothesis on the group $\prod_{H}G$, given the finite set
\begin{equation}\label{E sub g}
F_{\,\prod_{H}G}
\defeq \bigcup_{\substack{ x \in \text{proj}_{\prod_{H}G}(F_0)\\ b\in B}} \theta_{b^{-1}}(x),
\end{equation}
there exist a  group  $(K,d)\in \mathcal C$ of ${\rm diam}(K)\leq1$ and a function $\varphi:\prod_{H} G \to K$ with  $\varphi(1)=1$ 
that is $(F_{\,\prod_{H}G},\varepsilon/3,d)$-multiplicative and
$(F_{\,\prod_{H}G},\varepsilon/3,d)$-free for a function $\rho$ independent of $K$.
Define 
\begin{align*}
\Phi: G \wrwr H&\to K\wr_B Sym(B)\\
(x,h)&\mapsto \left((\varphi \theta_{b^{-1}}(x))_{b \in B}, \sigma(h)\right) 
\end{align*}

{\bf Claim:} $\Phi$ is $(F_0,\varepsilon,\tilde d)$-multiplicative and $(F_0,\varepsilon ,\tilde d)$-free for the function 
$\tilde\rho(\varepsilon)=\min(1-\varepsilon/3, \rho(\varepsilon/3))$.

We will first prove that $\Phi$ is $(F_0,\varepsilon,\tilde{d})$-multiplicative. To that end, take $(x,h),(x',h')\in F_0$.
On the one hand
\begin{align*}
\Phi\left((x,h)(x',h')\right)=\Phi\left(x\theta_h(x'),hh'\right)= \left(\left(\varphi\left(\theta_{b^{-1}}(x)\theta_{b^{-1}h}(x')\right)\right)_{b\in B},\sigma(hh')\right)
\end{align*}
and on the other hand
\begin{align*}
\Phi(x,h)\Phi(x',h') &=\Big(\left(\varphi\left(\theta_{b^{-1}}(x)\right)\right)_{b\in B} , \sigma(h)\Big)
\Big(\left(\varphi\left(\theta_{b^{-1}}(x')\right)\right)_{b\in B} , \sigma(h')\Big )\\
&=\left(\left(\varphi\left(\theta_{b^{-1}}(x)\right)\varphi\left(\theta_{(\sigma(h)^{-1}b)^{-1}}(x')\right)\right)_{b\in B},\sigma(h)\sigma(h')\right).
\end{align*}
Then
 \begin{align}\label{Suma Grande}
  &\tilde d \big(\Phi((x,h)(x',h')),\Phi(x,h)\Phi(x',h') \big)
=d_{\rm{Hamm}}(\sigma(hh'),\sigma(h)\sigma(h')) \  \ \ \ + \\
&\ \ \ \ +\,\,\frac{1}{|B|}  \sum_{\substack{b\in  B \\ \sigma(hh')b =\sigma(h)\sigma(h')b}}
d\left(\varphi\left(\theta_{(\sigma(hh')b)^{-1}}(x)\theta_{(\sigma(h)\sigma(h')b)^{-1}h}(x')\right),\varphi\left(\theta_{(\sigma(hh')b)^{-1}}(x)\right)
\varphi\left(\theta_{(\sigma(h')b)^{-1}}(x')\right)\right).\nonumber
\end{align}
Since $\sigma$ is $(F_H,\varepsilon/3, d_{\rm{Hamm}})$-multiplicative, it follows that 
$d_{\rm{Hamm}}(\sigma(hh'),\sigma(h)\sigma(h')) <\varepsilon/3$.
It only remains to bound the second summand of \eqref{Suma Grande}.
 To that end, we will partition the set $B$ in two disjoint subsets, one in which we can control the
 sum because all its summands are small, and another in which we can control the sum because the subset itself is small and ${\rm diam}(K)\leq 1$.

On the one hand, if  $h\sigma(h')b\in B$ then by the definition of $\sigma$ given in \eqref{def of sigma}, 
 we have that $\sigma(h)\sigma(h')b=h\sigma(h')b$, and so $(\sigma(h)\sigma(h')b)^{-1}h=(\sigma(h')b)^{-1}$.
Since, by \eqref{E sub g}, $\theta_{(\sigma(hh')b)^{-1}}(x), \theta_{(\sigma(h')b)^{-1}}(x')\in F_{\,\prod_{H}G}$,
we conclude that 
for all  $b\in \sigma(h')^{-1}(B\cap h^{-1}B)$ 
\begin{equation} \label{cota en sumandos chicos}
d\left(\varphi(\theta_{(\sigma(hh')b)^{-1}}(x)\theta_{(\sigma(h)\sigma(h')b)^{-1}h}(x')),\varphi\left(\theta_{(\sigma(hh')b)^{-1}}(x)\right)
\varphi(\theta_{(\sigma(h')b)^{-1}}(x'))\right)<\varepsilon/3.
\end{equation}
On the other hand, if $b\in \sigma(h')^{-1}(B\setminus h^{-1}B)$, this subset  is small because by \eqref{folner}
we have that 
\begin{equation}\label{cota en conjunto chico}
|\sigma(h')^{-1}(B\setminus h^{-1}B)|=|B\setminus h^{-1}B|<\frac{\varepsilon}{6}|B|.
\end{equation}
Partitioning the set $B$ in the disjoint subsets $ \sigma(h')^{-1}(B\cap h^{-1}B)$ and $\sigma(h')^{-1}(B\setminus h^{-1}B)$ and
replacing \eqref{cota en sumandos chicos} and \eqref{cota en conjunto chico} in \eqref{Suma Grande} gives
$\tilde d \big(\Phi((x,h)(x',h')),\Phi(x,h)\Phi(x',h') \big) <5/6\,\varepsilon.$

Let us now prove that $\Phi$ is $(F_0,\varepsilon,\tilde d)$-free for the function $\tilde\rho$.
If $h\in F_H\setminus\{1\}$ then
$$\tilde d(\Phi(x,h),1) \geq d_{\text{Hamm}}(\sigma(h),1) \geq 1-\varepsilon/3\geq \tilde\rho(\varepsilon).$$
We are left to show that $\Phi$ is $(F_0,\varepsilon,\tilde d)$-free for the function $\tilde\rho$ in the case when $(x,1)\in F_0\setminus\{1\}$.
Since $x\neq 1$, then $\theta_{b^{-1}}(x) \in F_{\,\prod_{H}G} \setminus \{1\}$. Hence
 $\tilde d (\Phi(x,1),1) =
 \frac{1}{|B|}\sum_{b\in B} d(\varphi  \theta_{b^{-1}}(x),1) \geq
 \rho(\varepsilon/3)\geq\tilde\rho(\varepsilon).$
\end{proof}
\begin{remark} In the proof we never used that $\theta$ is the shift action of $H$ on $\prod_H G$. In fact the very same statement and proof hold true
if we replace $\prod_H G$ by a group, call it $E$, and we replace $G\wrwr H$ by a semi-direct product $E\rtimes_\theta H$. As it was discussed in the introduction, the Kaloujnine-Krasner theorem implies that both statements are equivalent. 
In Section \ref{S5}, we  explain why even a stronger version of the Kaloujnine-Krasner theorem does not seem to be suitable to prove the stronger version of Corollary \ref{coro on extensions} 
 when the assumptions of $N$ being a normal subgroup of $G$ and $G/N$ being amenable are replaced by the hypothesis of $G$ having a co-amenable subgroup $H$. In that section, we instead prove a more technical variant of Proposition \ref{weak sofic permanence} which must avoid wreath products and semi-direct products.

\end{remark}
\section{Group monomorphisms which control the metrics and proof of Theorem \ref{main teorema}}\label{S4}
In order to prove  Theorem \ref{main teorema},  we apply Proposition  \ref{weak sofic permanence} to the  classes $\mathcal C$ of Examples \ref{ejemplos} together with group homomorphisms built in each case to carry the metric structure from $(K\wr_B Sym( B),\tilde d)$ with $(K,d)\in \mathcal C$ to a group in the class $\mathcal C$ in a controlled manner. 

\begin{proof}[Proof of Theorem \ref{main teorema}(\ref{caso weak sofic})]
If $G$ is weakly sofic, then $\prod_H G$ is weakly sofic and  it satisfies the hypothesis of  Proposition \ref{weak sofic permanence}
when  $\rho$ is equal to the constant function $1$.
Hence, given $\varepsilon>0$ and a finite set $F\subseteq G\wrwr H$, 
there exist a finite set $B\subseteq H$, a  finite group $K$ with a bi-invariant metric $d$, ${\rm diam}(K)=1$, and a function 
$\Phi: G\wrwr H\to K\wr_B Sym( B)$ that is $(F,\varepsilon,\tilde d)$-multiplicative and $(F,\varepsilon, \tilde d)$-free for 
$\tilde \rho(\varepsilon)=1-\varepsilon/3\geq 2/3$. Since $K\wr_B Sym( B)$ is a finite group, this concludes the proof.
\end{proof}
\begin{remark} The proof is constructive in the sense that starting with a given weakly sofic approximation of $G$, 
one can explicitly construct a weakly sofic approximation of  $\prod_H G$, and the proof provides a weakly sofic approximation of $G\wrwr H$.
This is also the case in the remaining instances of Theorem \ref{main teorema}.
\end{remark}
\subsection{The sofic case.}
We will need the next lemma, a variant of which is present in \cite{HAY18}.
\begin{lemma} \label{isometria sofic}
Let $A,B$ be finite sets. The function
\begin{align*}
\psi: (Sym(A) \wr_B Sym(B),\tilde{d}) &\to (Sym(A\times B),d_{\rm{Hamm}})\\
		\psi(\alpha,\beta)(a,b)&:=(\alpha_{\beta(b)}(a),\beta(b))
\end{align*}
is an isometric group homomorphism.
\end{lemma}
\begin{proof} 
For $(\alpha,\beta),(\alpha',\beta')\in Sym(A) \wr_B Sym (B)$ and $(a,b)\in A\times B$, 
the following identities show that $\psi$ is a homomorphism:
\begin{align*}
&\psi(\alpha,\beta)\psi(\alpha',\beta')(a,b)= \psi(\alpha,\beta)(\alpha'_{\beta'(b)}(a),\beta'(b))= (\alpha_{\beta(\beta'(b))}(\alpha'_{\beta'(b)}(a)),\beta\beta'(b));
\\
&\psi(\alpha (\beta\cdot\alpha'), \beta\beta')(a,b)= ((\alpha (\beta\cdot\alpha'))_{\beta\beta'(b)}(a),\beta\beta'(b))= (\alpha_{\beta\beta'(b)}\alpha'_{\beta'(b)}(a),\beta\beta'(b)).
\end{align*}
 $\psi$ is an isometry because
\begin{align*}
d&_{\rm{Hamm}}(\psi(\alpha,\beta),\psi(\alpha',\beta'))= \frac{1}{|A||B|}|\{(a,b)\in A\times B : (\alpha_{\beta(b)}(a),\beta(b)) \neq (\alpha'_{\beta'(b)}(a),\beta'(b))\}|\\
&=\frac{1}{|B|}\sum_{\{b\in B:\beta(b)\neq \beta'(b)\}}\frac{1}{|A|}|\{a\in A : (\alpha_{\beta(b)}(a),\beta(b)) \neq (\alpha'_{\beta'(b)}(a),\beta'(b))\}|\\
 & \ \ \ \ \ \ \ +\frac{1}{|B|}\sum_{\{b\in B:\beta(b)=\beta'(b)\}}\frac{1}{|A|}|\{a\in A : \alpha_{\beta(b)}(a) \neq \alpha'_{\beta(b)}(a)\}| \\
&= d_{{\rm Hamm}}(\beta,\beta') + \frac{1}{|B|} \sum_{\{b\in B:\beta(b)=\beta'(b)\}} d_{{\rm Hamm}}(\alpha_{\beta(b)},\alpha^{\prime}_{\beta(b)}) = \tilde d ((\alpha,\beta),(\alpha',\beta')).\qedhere
\end{align*}
\vspace{-0.3cm}\phantom\qedhere
\end{proof}

\begin{proof}[Proof of Theorem \ref{main teorema}(\ref{caso  sofic})] 
If $G$ is sofic, then $\prod_H G$ is sofic and  it  satisfies the hypothesis of  Proposition \ref{weak sofic permanence} 
 when $\rho(\varepsilon')=1-\varepsilon'$.
Hence, given $\varepsilon>0$ and a finite set $F\subseteq G\wrwr H$, 
 there exist a finite set $B\subseteq H$, a  finite permutational group $Sym(A)$ endowed with the normalized Hamming distance and a function 
$\Phi: G\wrwr H\to Sym(A)\wr_B Sym( B)$ that is $(F,\varepsilon,\tilde d)$-multiplicative and $(F,\varepsilon,\tilde d)$-free, 
for $\tilde\rho(\varepsilon)=1-\varepsilon/3\geq 1-\varepsilon$. Lemma \ref{isometria sofic} implies that 
$\psi\circ\Phi:G\wrwr H\to Sym(A\times B)$ is a $(F,\varepsilon,d_{\rm{Hamm}})$-sofic approximation of $G\wrwr H$.
\end{proof}
\subsection{The linear sofic case.}
Take a finite set  $B$. 
Consider $M_{|B|n}(K)$ and identify it with $M_{|B|}(M_n(K))$.
Hence, for $A\in M_{|B|n}(K)$ and $b',b\in B$,  $A_{(b',b)}\in M_n(K)$ denotes  the block entry of $A$ at the coordinates $(b',b)$.
Define 
\begin{align}\label{definicion de psi}
\psi: M_n(K) \wr_B Sym(B) &\to M_{|B|n}(K)\\
			\psi (U,\tau)_{(b',b)} &=\begin{cases}
										0 & \text { if } b' \neq \tau(b)\\
										U_{\tau(b)} & \text{ if } b'= \tau(b).
							\end{cases}\nonumber
\end{align}  
\begin{lemma}\label{Lemma desigualdad linear sofic}
$\psi$  
is a group homomorphism between  $GL_{n}(K) \wr_B Sym(B)$ and $GL_{|B|n}(K)$
and satisfies that 
\begin{equation}\label{ecuacion desigualdad linear sofic}
\frac{1}{2}\tilde d\left((U,\tau),(U',\tau')\right)\leq d_{{\rm rk}}\left(\psi(U,\tau),\psi(U',\tau')\right)\leq \tilde d\left((U,\tau),(U',\tau')\right).
\end{equation}
\end{lemma}
\begin{proof}
It is obvious that $\psi $ maps  $GL_n(K) \wr_B Sym(B)$ to $GL_{|B|n}(K)$. 
A routine matrix computation shows that it is a group homomorphism.
In order to bound $d_{{\rm rk}}\left(\psi(U,\tau),\psi(U',\tau')\right)$,
first observe that the matrix $\psi(U,\tau)-\psi(U',\tau')\in M_{|B|}(M_n(K))$
has at most two nonzero block-entries in each column and in each row. Moreover, the only
columns with one nonzero block-entry are the columns $b\in B$ for which $\tau(b)=\tau'(b)$ and the only rows with one nonzero block-entry are 
the rows $\tilde b\in B$ such that $\tau^{-1}(\tilde b)=\tau'^{-1}(\tilde b)$. With this in mind, 
let $A_1$ be  the submatrix of  $\psi(U,\tau)-\psi(U',\tau')$ obtained after removing the rows and columns with 
exactly two nonzero block-entries and let $A_2$ be the submatrix of $\psi(U,\tau)-\psi(U',\tau')$ obtained after
removing the rows and columns with exactly one nonzero block-entry.
 Then 
\begin{equation}\label{ec1}
d_{{\rm rk}}\left(\psi(U,\tau),\psi(U',\tau')\right)
=\frac{1}{|B|n}{\rm rank} (\psi(U,\tau)-\psi(U',\tau'))=\frac{1}{|B|n}{\rm rank} (A_1)+\frac{1}{|B|n}{\rm rank}(A_2),
\end{equation}
and 
\begin{equation}\label{ec2}
\frac{1}{|B|n}{\rm rank}(A_1)= 
\frac{1}{|B|}\sum_{\substack{b\in B\\ \tau(b)=\tau'(b)}} \frac{1}{n}{\rm rank} \left(U_{\tau(b)} -U'_{\tau(b)}\right)=
\frac{1}{|B|}\sum_{\substack{b\in B\\ \tau(b)=\tau'(b)}}d_{{\rm rk}} \left(U_{\tau(b)},U'_{\tau(b)}\right).
\end{equation}
If we regard
$A_2$ as a block matrix
in $M_{|\{b \in B: \tau(b)\neq\tau'(b)\}|}(M_n(K))$ then all its nonzero entries 
are in $GL_n(K)$; each column $b$ has exactly two nonzero entries, 
the ones corresponding to the rows $\tau(b)$ and $\tau'(b)$; and each  row $\tilde b$ has exactly two nonzero  entries,
the ones corresponding to the columns  $\tau^{-1}(\tilde b)$ and $\tau'^{-1}(\tilde b)$.
 Then  
 \begin{equation}\label{ec3}
 \frac{1}{2}n\big |\{b\in B:  \tau(b)\neq \tau'(b)\}\big|\leq {\rm rank}( A_2)\leq n\big|\{b\in B:  \tau(b)\neq \tau'(b)\}\big |,
 \end{equation}
 where the first inequality follows from the simple fact that if a matrix $A\in M_{r}(K)$ 
 has exactly two nonzero entries in each column and in each row, then ${\rm rank}(A)\geq r/2$.
  Replacing \eqref{ec2} and \eqref{ec3}
 in \eqref{ec1} yield the desired result.
 \end{proof}

\begin{proof} [Proof of Theorem \ref{main teorema}(\ref{caso linear sofic})]

If $G$ is linear sofic, then $\prod_H G$ is linear sofic, and  it  satisfies the hypothesis of  Proposition \ref{weak sofic permanence} 
 when $\rho(\varepsilon')=1/4-\varepsilon'$. The proof proceeds as in the sofic case.
\end{proof}
\subsection{The hyperlinear case.}\label{subsection hiperlinear}
Identify  $\mathcal B(\CH)$ with $M_n(K)$  (with $K=\mathbb R$  or $K=\mathbb C$) via the matrix representation in an  orthonormal basis $\{v_1,\ldots,v_n\}$. 
It is clear that the function  $\psi$ defined in \eqref{definicion de psi} can be regarded as 
$\psi:  \mathcal B(\CH) \wr_B Sym(B) \to  \mathcal B(\bigoplus_{B}\CH)$ and that 
$\|\psi(U,\tau)\|_2^2=\frac{1}{|B|}\sum_{b\in B}\|U_b\|_2^2$.
Recall that the diameter of the unitary group in the Hilbert-Schmidt metric is ${2}$.
Then the appropriate metric in 
 $\mathcal U(\CH) \wr_B Sym(B)$ is the one obtained by scaling
the second summand in  \eqref{metric in permut. wreath product} by $1/2$. We still call this metric $\tilde d$.

\begin{lemma}\label{desigualdad de hiperlineal}
$\psi$  
is a group homomorphism between  $\mathcal U(\mathcal H) \wr_B Sym(B)$ and $\mathcal U(\bigoplus_B\mathcal H)$
and satisfies that 
\begin{equation}\label{desigualdad hiperlin}
\tilde d\left((U,\tau),(U',\tau')\right)\leq d_{{\rm HS}}\left(\psi(U,\tau),\psi(U',\tau')\right)\leq 2\sqrt{\tilde d\left((U,\tau),(U',\tau')\right)}.
\end{equation}
\end{lemma}

\begin{proof}It is obvious that $\psi $ maps  $\mathcal U(\mathcal H) \wr_B Sym(B)$ to $\mathcal U(\bigoplus_B\mathcal H)$.
 Moreover, on the one hand
\begin{align*}
\|\psi(U,\tau)-\psi(U',\tau')\|_2^2
&= \frac{1}{|B|} \sum_{\substack{b\in B\\ \tau(b)\neq\tau'(b)}}
\Big\|
U_{\tau(b)}\Big\|_2^2
+\left\|U'_{\tau'(b)}
\right\|_2^2\\
&\,\,\,\,\, \ \ \ + 
\frac{1}{|B|} \sum_{\substack{b\in B\\ \tau(b)=\tau'(b)}}
 \left\|U_{\tau(b)}-U'_{\tau(b)}\right\|_2^2 \\
&=2 d_{\rm{Hamm}}(\tau,\tau') + \frac{1}{|B|}\sum_{\substack{b\in B\\ \tau(b)=\tau'(b)}}d_{{\rm HS}}\left(U_{\tau(b)},U'_{\tau(b)}\right)^2 \\
&\leq4\tilde d(\psi(U,\tau),\psi(U',\tau')),
\end{align*}
on the other hand, by the Cauchy-Schwarz inequality, we have that 
\begin{align*}
&2 d_{\rm{Hamm}}(\tau,\tau') + \frac{1}{|B|}\sum_{\substack{b\in B\\ \tau(b)=\tau'(b)}}d_{{\rm HS}}\left(U_{\tau(b)},U'_{\tau(b)}\right)^2
\geq  2 d_{\rm{Hamm}}(\tau,\tau') + \left(\frac{1}{|B|}\sum_{\substack{b\in B\\ \tau(b)=\tau'(b)}}d_{{\rm HS}}\left(U_{\tau(b)},U'_{\tau(b)}\right)\right)^2\\
&\geq 2d_{\rm{Hamm}}(\tau,\tau')^2 + 2\left(\frac{1}{2|B|}\sum_{\substack{b\in B\\ \tau(b)=\tau'(b)}}d_{{\rm HS}}\left(U_{\tau(b)},U'_{\tau(b)}\right)\right)^2
\geq \tilde d(\psi(U,\tau),\psi(U',\tau'))^2.\qedhere
\end{align*}
\vspace{-0.3cm}\phantom\qedhere
\end{proof}

\begin{remark}\, \label{about radulescu theorem}
According to the original definition  given by R\u{a}dulescu in 2000, a  group $G$ is {\it hyperlinear} if it embeds in $\mathcal U(R^{\omega})$, the unitary group of the ultrapower of the
hyperfinite $\rm{II}_1$ factor $R$. R\u{a}dulescu showed in  \cite{MR2436761} (see also \cite[Proposition 7.1]{MR2072092} for a simpler proof) that this is the same as saying that the group von Neumann algebra $L(G)$  satisfies the Connes' embedding problem, namely that  $L(G)$ embeds in $R^{\omega}$. This, in turn, implies that the embedding of $G$  in $\mathcal U(R^{\omega})$ can be taken to be an orthogonal embedding, (as a mater of fact, the proof of
 R\u{a}dulescu's theorem boils down to  show exactly this). Since unitaries in $R$ can be approximated with $n\times n$ unitary matrices with $n$
sufficiently large,  it follows that $G$ is hyperlinear if and only if $G$ embeds orthogonally in an ultraproduct of unitary matrices endowed with the Hilbert-Schmidt norm. When $G$ is countable, translating from the language of ultraproducts to a finitary version yields the definition of hyperlinear group we gave in Example \ref{ejemplos}\eqref{hiperlin}, with the notion of $(F,\varepsilon,d_{{\rm HS}})$-trace preserving defined in \eqref{trace preserving}.
By means of the Cauchy-Schwartz inequality, it is clear that $(F,\varepsilon,d_{{\rm HS}})$-trace preserving is equivalent to
\begin{equation}
\label{alternative 4}
 \sqrt 2-\varepsilon< \|\phi(g)-1\|_{2}<\sqrt 2+\varepsilon \,\,\text{ for
all }g\in F\setminus \{1\}.
\end{equation}
These two inequalities combined just reflect that $\phi(g)$ can be taken to be ``almost orthogonal'' to $1$ in the Hilbert-Schmidt norm. After removing the second inequality in \eqref{alternative 4}, one gets that $\phi$ is $(F,\varepsilon,d_{{\rm HS}})$-free for the function $\rho(\varepsilon):=\sqrt 2-\varepsilon$ and this is what some articles write as part of the definition of hyperlinear group.  Both definitions are equivalent due to R\u{a}dulescu's theorem, 
 however,  a hyperlinear approximation that is $(F,\varepsilon,d_{{\rm HS}})$-trace preserving carries 
more information than one that is only $(F,\varepsilon,d_{{\rm HS}})$-free for the function $\rho(\varepsilon)=\sqrt 2-\varepsilon$.
\end{remark}
 A consequence of the previous remark is that one could use the characterization of hyperlinear groups in terms of  $(F,\varepsilon,d_{{\rm HS}})$-freeness to deduce Theorem \ref{main teorema}(\ref{caso hyperlineal}) from Proposition \ref{weak sofic permanence} and Lemma \ref{desigualdad de hiperlineal} following the same strategy employed in the sofic and linear sofic cases. We leave this approach as an exercise.
 The disadvantage with this approach is that if one starts with a hyperlinear approximation of $G$ that is $(F,\varepsilon,d_{{\rm HS}})$-trace preserving, the proof would only  guarantee a hyperlinear approximation of $G\wrwr H$ that is $(F,\varepsilon,d_{{\rm HS}})$-free for the function $\rho(\varepsilon)=\sqrt 2-\varepsilon$.
However, minor adjustments to the proof of Proposition \ref{weak sofic permanence}
are enough to provide a hyperlinear approximation of $G\wrwr H$ that is $(F,\varepsilon,d_{{\rm HS}})$-trace preserving.
We sketch them in the next proof. 
\begin{proof}[Proof of Theorem \ref{main teorema}(\ref{caso hyperlineal})]\label{dem caso hl}
Consider $F_0$, $F_H$, $B$, $\sigma$ and $F_{\,\prod_{H}G}$ 
as in the proof of Proposition \ref{weak sofic permanence}, 
but in this case we take $\varepsilon^2/24$ in \eqref{folner} 
so that $\sigma$ becomes an $(F_H,\varepsilon^2/12)$-sofic approximation of $H$.
Since  $\prod_H G$ is hyperlinear, 
there exist a finite dimensional Hilbert space $\mathcal H$ and a function 
$\varphi: \prod_H G\to \mathcal U(\mathcal H)$ with $\varphi(1)=1$ that is 
$(F_{\,\prod_{H}G},\varepsilon^2/12, d_{\rm{HS}})$-multiplicative 
and $(F_{\,\prod_{H}G},\varepsilon^2/12,d_{\rm{HS}})$-trace preserving.

Keeping in mind that the metric in $\mathcal U(\mathcal H)\wr_B Sym( B)$ is 
obtained by multiplying the second summand in  equation \eqref{metric in permut. wreath product} by $1/2$,
the same proof of Proposition \ref{weak sofic permanence}
shows that the function 
$\Phi: G\wrwr H\to \mathcal U(\mathcal H)\wr_B Sym( B)$ 
 is $(F_0,\varepsilon^2/4,\tilde d)$-multiplicative. Then, the second inequality in
 \eqref{desigualdad hiperlin}
  implies that 
$\psi\circ\Phi:G\wrwr H\to \mathcal U(\bigoplus_B\mathcal H)$ is $(F_0,\varepsilon,d_{\rm{HS}})$-multiplicative.
It remains to show that $\psi\circ\Phi$ is $(F_0,\varepsilon,d_{\rm{HS}})$-trace preserving.
The basic observation that the trace of a block-matrix is equal to the sum of the traces of its block-diagonal entries, yields that
$tr (\psi((U_{b})_{b\in B},\tau))=\frac{1}{|B|}\sum_{b\in B:\tau(b)=b} tr(U_{b}).$
Hence 
$$\left|tr (\psi\circ\Phi(x,h))\right | 
= 
\frac{1}{|B|}\left |\sum_{b\in B:\sigma(h)b=b}tr(\varphi \theta_{b^{-1}}(x))\right | 
\leq \frac{1}{|B|}\left |\{b\in B:\sigma(h)b=b\}\right |=1-d_{{\rm Hamm}}(\sigma(h),1).$$
It follows that $\left | tr (\psi\circ\Phi(x,h))\right | <\varepsilon^2/12<\varepsilon$, whenever $ h\in F_H\setminus\{1\}$.
On the other hand, if $(x,1)\in F_0\setminus\{1\}$ then $\theta_b^{-1}(x)\in F_{\,\prod_{H}G}\setminus\{1\}$. Thus
$\left| tr (\psi\circ\Phi(x,1))\right |=\frac{1}{|B|}\left |\sum_{b\in B}tr(\varphi \theta_{b^{-1}}(x))\right | < \varepsilon^2/12<\varepsilon.$
\end{proof}
\section{The case of co-amenable subgroups}\label{S5}
The proof of Elek and Szab\'o in \cite{ELE03} showing that extensions of the form sofic-by-amenable are sofic 
can  be adapted easily to show that if a group $G$ has a sofic and co-amenable subgroup $H$, then $G$ is sofic. 
Moreover, there is a seemingly less well-known theorem of Kaloujnine and Krasner  \cite{MR49892}, stating that if $H<G$, 
then  $G$ embeds  in $H\wrwr \left(G/\text {core}(H)\right)$, (see also  \cite[Theorem 2.3.1]{MR2674854}).
Let us emphasize that $H$ does not need to be normal in none of these statements. 
Unfortunately, we do not know how to use this theorem of Kaloujnine and Krasner to show the
theorem of Elek and Szab\'o for co-amenable groups. 
 For instance it is false that if $H$ is co-amenable on $G$, 
 then $\text {core}(H)$ is co-amenable in $G$, (which is the same as saying that the group $G/ \text {core}(H)$ is amenable) so it seems hard to have some control
 on $H\wrwr\left(G/\text {core}(H)\right)$.
 As an example, Monod and Popa showed in \cite{MR1999183} that for any countable group $Q$, the group $H:=\oplus_{n\geq 0} Q$ is 
 co-amenable in $G:=Q\wr \mathbb Z$, the restricted wreath product of $Q$ by $\mathbb Z$. 
 Since ${\rm core}(H)=\{1\}$, taking any non-amenable $Q$, it follows that $G/{\rm core}(H)$ is non-amenable.
 Moreover, in this example, the theorem of Kaloujnine and Krasner is the trivial 
 statement $G<H\wrwr G$, and it is unknown in general whether $H\wrwr G$ is sofic when $Q$ (and hence $H$ and $G$) are sofic. 
 
In spite of this obstruction, we are still able to use the work we have done in the previous sections to show in a unified way that if $H$ is co-amenable in $G$, and $H$ belongs to one of the four classes $\mathcal C$ in the Examples \ref{ejemplos}, then $G$ is in the same class $\mathcal C$ as $H$. 
We achieve this by  replacing each time we used Proposition \ref{weak sofic permanence} in  Section \ref{S4} with the next statement.
\begin{proposition}
Let $G$ be  group and
let $H$ be a  subgroup with the property that for every $\varepsilon'>0$ and for every finite subset $F_{H}$,  
there exist a group  $(K,d)\in \mathcal C$ of diameter bounded by $1$ and a function 
$\varphi:H\to K$ with $\varphi(1)=1$ that is $(F_{H},\varepsilon', d)$-multiplicative 
and $(F_{H},\varepsilon', d)$-free, for certain function $\rho$, independent of $K$.
Moreover, suppose that $H$ is co-amenable in $G$, namely suppose that the left regular action of $G$ on the left cosets $G/H$ is amenable.

Then, given $\varepsilon>0$ and a finite set $F\subseteq G$, there exist a finite set $ B\subseteq G$, 
a group $(K,d)\in \mathcal C$ and a function 
$\Phi: G\to K\wr_B Sym( B)$ that is $(F,\varepsilon,\tilde d)$-multiplicative 
and $(F,\varepsilon, \tilde d)$-free,
for the function $\tilde\rho(\varepsilon):=(1-\varepsilon/3) \rho(\varepsilon/3)$.
\end{proposition}
\begin{proof}
Let $F\subseteq G$ be a finite subset and let $\varepsilon>0$. Since the action of $G$ on $G/H$ defined by $g. xH := gxH$ is amenable,
 by the  F\o lner's condition (see, for instance, \cite{MR2303529,MR333068}), there exists a finite subset $\overline{B}\subseteq G/H$ such that
  \begin{equation} \label{folnerB}
  \frac{|f.\overline B \bigtriangleup \overline B|}{|\overline B|}\leq \varepsilon/4,\text{ for all } f\in F\cup F^{-1}\cup F^2\cup (F^{-1})^2.
  \end{equation}
Let $\sigma: G \to Sym (\overline{B})$ be defined as
\begin{align*}\label{theta amen}
 \sigma(g)\bar b &:= \begin{cases}  g.\bar{b} &\text{ if } g.\bar{b} \in\overline B \\
        \gamma_g {(g.\bar{b})} &\text{ if not, where }  
       \gamma_g : g\overline{B} \setminus \overline{B} \to \overline{B} \setminus g\overline{B} \text{ is a fixed bijection}. 
         \end{cases}   
\end{align*}
\begin{claim}\label{Hamm-multiplicative}
$\sigma$ is $(F,\varepsilon/4,d_{{\rm Hamm}})$-multiplicative.
\end{claim}
\begin{proof}[Proof of claim]Let $f_1,f_2 \in F$. Define  $S:=\{\overline{b}\in \overline{B} :f_2.\overline{b}\in \overline{B}\}\cap \{\overline{b}\in \overline{B}:
f_1f_2.\overline{b} \in \overline{B}\}$. Note that if $\overline{b}\in S$, then $\sigma(f_1f_2)\bar b = \sigma(f_1)\sigma(f_2)\bar b$.
Moreover,
\begin{align*}
 |\overline{B}|&\geq |\{\overline{b}\in \overline{B} :f_2.\overline{b}\in \overline{B}\}\cup \{\overline{b}\in \overline{B} :f_1f_2.\overline{b}\in \overline{B}\}|\\
&\geq |\{\overline{b}\in \overline{B} :f_2.\overline{b}\in \overline{B}\}|+| \{\overline{b}\in \overline{B} :f_1f_2.\overline{b}\in \overline{B}\}|
- |\{\overline{b}\in \overline{B} :f_2.\overline{b}\in \overline{B}\}\cap \{\overline{b}\in \overline{B} :f_1f_2.\overline{b}\in \overline{B}\}|\\
&\geq 2(1-\varepsilon/8)|\bar B|
- |S|,
\end{align*}
so $\frac{|\overline B\setminus S|}{|\overline B|} \leq   \varepsilon/4$, (in the third inequality we used that, by construction, \eqref{folnerB}
is valid when $f=f_2^{-1}f_1^{-1}$, and that \eqref{folnerB} implies that for every $f\in F^{-1}\cup (F^{-1})^2$, 
$|f\overline B\cap\overline B|\geq (1-\varepsilon/8)|\overline B|$). 
  Then
\begin{align*}
d_{{\rm Hamm}}(\sigma(f_1)&\sigma(f_2),\sigma(f_1f_2))=\frac{1}{|\overline{B}|} \left\vert\left\{ \overline b\in \overline{B}: \sigma(f_1f_2)\overline{b} \neq \sigma(f_1)\sigma(f_2)\overline{b}\right\} \right\vert\\
&=\frac{1}{|\overline{B}|} \left\vert\left\{ \overline b\in S: \sigma(f_1f_2)\overline{b} \neq \sigma(f_1)\sigma(f_2)\overline{b}\right\}\right\vert + \frac{1}{|\overline{B}|} \left\vert\left\{ \overline b\in \overline{B}\setminus S: \sigma(f_1f_2)\overline{b} \neq \sigma(f_1)\sigma(f_2)\overline{b}\right\}\right\vert\\
&\leq 0+\frac{|\overline{B}\setminus S|}{|\overline{B}|}\leq \varepsilon/4.
\qedhere
\end{align*}
\vspace{-0.3cm}\phantom\qedhere
\end{proof}
Let $\pi:G \to G/H$ be the quotient map and let $\tau: G/H \to G$ be a section of it. 
 Call $B:=\tau(\overline{B})$.  Let $F_H := \left(B^{-1} F B\right) \cap H$.
  By  hypothesis,  there exist a  group $(K,d)\in \mathcal C$ of diameter bounded by $1$
 and a function $\varphi: H \to K$  with $\varphi(1)=1$ that is  $(F_H,\varepsilon/3,d)$-multiplicative and $(F_{H},\varepsilon/3, d)$-free for certain function $\rho$, independent of $K$.
 Define
\begin{align}
G &\to K \wr_B Sym(B)\\
\nonumber \Phi(g) &:= \left( \left( \varphi\left( b^{-1}g \tau \pi(g^{-1}b)\right)\right)_{b\in B}, \tau\sigma(g)\pi \right).
\end{align}

$\Phi$ is well defined since  $b^{-1}g \tau \pi(g^{-1}b)\in H$ and it is clear that $\Phi(1)=1$. 
(The idea of using $\varphi\left( b^{-1}g \tau \pi(g^{-1}b)\right)$ is already present in \cite{ELE03}).
\\
{\bf Claim:} $\Phi$ is $(F, \varepsilon, \tilde{d})$-multiplicative and 
$(F, \varepsilon, \tilde{d})$-free for the function $\tilde\rho(\varepsilon)=(1-\varepsilon/3) \rho(\varepsilon/3)$.

Let us first prove that $\Phi$ is $(F, \varepsilon, \tilde{d})$-multiplicative.
 In order to shorter the length of the equation lines, call $\psi_b(g) :=\varphi\left( b^{-1}g \tau \pi(g^{-1}b)\right)$.  Let $f_1,f_2\in F$.
  Then
\begin{align*}
\Phi(f_1f_2) =  \left( \left( \psi_b(f_1f_2)\right)_{b\in B}, \tau\sigma(f_1f_2)\pi \right)= \left( \left( \varphi\left( b^{-1}f_1f_2 \tau\pi((f_1f_2)^{-1}b)\right)\right)_{b\in B}, \tau\sigma(f_1f_2)\pi \right)
\end{align*}
and
\begin{align*}
\Phi(f_1)\Phi(f_2) &=
\left( \left(\psi_b(f_1)\right)_{b\in B}, \tau\sigma(f_1)\pi\right)\left( \left(\psi_b(f_2)\right)_{b\in B}, \tau\sigma(f_2)\pi\right)
\\
&
=\left(\left(\psi_b(f_1) \psi_{\tau\sigma(f_1^{-1})\pi(b)}(f_2)\right)_{b\in B}, \tau\sigma(f_1)\sigma(f_2)\pi\right)
\end{align*}
Let $B_{f_1,f_2}:= \{b \in B: f_1^{-1}.\pi (b)\in  \overline{B}\}\cap\{b \in B: (f_1f_2)^{-1}.\pi (b)\in  \overline{B}\}$. Observe that if $b\in B_{f_1,f_2}$, then by the definition of $\sigma$, 
it holds that $\sigma(f_1^{-1})\pi (b)=f_1^{-1}.\pi (b)$  and
by the definition of the left regular action, it holds that $f_1^{-1}.\pi (b)= \pi(f_1^{-1} b)$ and $(f_1f_2)^{-1}.\pi (b)= \pi((f_1f_2)^{-1} b)\in  \overline{B}$. With this in mind, it follows that for all $b\in B_{f_1,f_2}$
\begin{align*}
 \psi_{\tau\sigma(f_1^{-1})\pi(b)}(f_2)&=\varphi\left( (\tau\sigma(f_1^{-1})\pi(b))^{-1}f_2 \tau\pi(f_2^{-1} \tau\sigma(f_1^{-1})\pi(b))\right)\\
&=
  \varphi\left( (\tau\pi(f_1^{-1} b))^{-1}f_2 \tau\pi(f_2^{-1} \tau \pi(f_1^{-1}b))\right)\\
 & =
   \varphi\left( (\tau\pi(f_1^{-1} b))^{-1}f_2 \tau\pi((f_1f_2)^{-1}b)\right).
 \end{align*} 
Since for all $b\in B_{f_1,f_2}$ it holds that $ b^{-1}f_1 \tau\pi(f_1^{-1}b)$ and $(\tau\pi(f_1^{-1} b))^{-1}f_2 \tau\pi((f_1f_2)^{-1}b)$ belong to $F_H$ and since $\varphi$ is $(F_H,\varepsilon/3,d)$-multiplicative, then
\[d\left(\psi_b(f_1) \psi_{\tau\sigma(f_1^{-1})\pi(b)}(f_2),\psi_b(f_1f_2)\right)<\varepsilon/3, \text{ for all }  b \in B_{f_1,f_2}.  \]
Hence, partitioning  $B$ with the subsets $(\tau\sigma(f_1f_2)\pi)^{-1}B_{f_1,f_2}$ and $B\setminus (\tau\sigma(f_1f_2)\pi)^{-1}B_{f_1,f_2}$ we get that
\begin{align*}\label{suma acotada s}
\tilde d(&\Phi (f_1)\Phi(f_2),\Phi(f_1f_2)) = d_{{\rm{Hamm}}}(\tau \sigma(f_1)\sigma(f_2)\pi, \tau\sigma(f_1f_2)\pi)+\\
&\frac{1}{| B|}\sum_{\substack {b\in B \\ \tau \sigma(f_1f_2)\pi(b)
 = 
 \tau\sigma(f_1)\sigma(f_2)\pi(b)}} d\left(\psi_{\tau \sigma(f_1f_2)\pi(b)}(f_1) \psi_{\tau\sigma(f_1^{-1})\sigma(f_1f_2)\pi(b)}(f_2),\psi_{\tau \sigma(f_1f_2)\pi(b)}(f_1f_2)\right)  \nonumber \\
&\leq  d_{{\rm{Hamm}}}(\tau \sigma(f_1)\sigma(f_2)\pi, \tau\sigma(f_1f_2)\pi)
+ 
\frac{1}{ |B|}\sum_{ b \in B_{f_1,f_2}}d\left(\psi_{b}(f_1) \psi_{\tau\sigma(f_1^{-1})\pi(b)}(f_2),\psi_{b}(f_1f_2)\right) \nonumber \\
&\,\,\,\,\,\,\,\,\,\,\,\,\,\,\,\,\,+ \frac{| B \setminus (\tau\sigma(f_1f_2)\pi)^{-1}B_{f_1,f_2}|}{| B|}.\nonumber
\end{align*}
Since $\tau$ is a bijection from $ \overline{B}$ to $B$ with inverse $\pi$, it follows that
 $$d_{{\rm Hamm}}(\tau\sigma(f_1)\sigma(f_2)\pi,\tau \sigma(f_1f_2)\pi)=d_{{\rm Hamm}}(\sigma(f_1)\sigma(f_2),\sigma(f_1f_2))\leq \varepsilon/4,$$
and that 
$$  \frac{| B \setminus (\tau\sigma(f_1f_2)\pi)^{-1}B_{f_1,f_2}|}{| B|}=\frac{| B \setminus B_{f_1,f_2}|}{| B|}=
\frac{|  \overline{B} \setminus \pi(B_{f_1,f_2})|}{|  \overline{B}|}.$$
Moreover, since $\pi(B_{f_1,f_2})=\{\bar b \in \overline{B}:f_1^{-1}.\bar b\in \overline{B}\}\cap \{\bar b \in\bar B:(f_1f_2)^{-1}.\bar b\in \overline{B}\}$, the first part of the proof of Claim \ref{Hamm-multiplicative} shows that $\frac{|  \overline{B} \setminus \pi(B_{f_1,f_2})|}{|  \overline{B}|}\leq \varepsilon/4$.
 Hence

$$\tilde d(\Phi (f_1)\Phi(f_2),\Phi(f_1f_2))<\varepsilon/4+\varepsilon/3+\varepsilon/4<\varepsilon.$$

In order to prove  that $\Phi$ is $(F,\varepsilon,\tilde d)$-free for the function $\tilde\rho(\varepsilon)=(1-\varepsilon/3)\rho(\varepsilon/3)$, take $f\in F\setminus \{1\}$ and define $B_f:=\{ b \in B : \sigma(f)\pi (b) = f.\pi(b)=\pi(fb)\}$. It holds that $|B_f|=|\{\bar b\in \overline B: f.\bar b\in \overline B\}|=|f^{-1}.\overline B\cap \overline B|\geq (1-\varepsilon/8)|B|$.
Observe that if $b\in B_f$ and $\tau \sigma(f)\pi(b)=b$ then  $\tau \pi(f^{-1}b)=b$.
 Also observe that 
 for all $b\in B$ such that $b^{-1}f b\in H$,  it holds that $d(\varphi(b^{-1}f b) ,1)>\rho(\varepsilon/3)$. 
 Moreover, since ${\rm diam}(K)\leq 1$ it follows that $\rho(\varepsilon)\leq 1$ for all $\varepsilon$.
With all this in mind, we proceed to bound $\tilde{d}(\Phi(f),1)$.
\begin{align*}
\tilde{d}(\Phi(f),1)
&=  d_{{\rm{Hamm}}}(\tau \sigma(f)\pi,1)+\frac{1}{ |B|}\sum_{\substack{ b \in B\\ \tau\sigma(f)\pi( b )=b}}d(\varphi(b^{-1}f\tau\pi(f^{-1}b)) ,1)\\
&\geq\frac{1}{|B|} |\{ b \in B_f: \tau\sigma(f)\pi( b) \neq b\}|+
   \frac{1}{| B|}\sum_{\substack{b \in B_{f}\\ \tau\sigma(f)\pi( b )=b}}d(\varphi(b^{-1}f b) ,1)\\
&\geq\frac{1}{|B|} |\{ b \in B_f: \tau\sigma(f)\pi( b) \neq b\}| + \frac{\rho(\varepsilon/3)}{|B|}| \{b\in B_f: \tau\sigma(f)\pi( b) = b\}|\\
&=\frac{1}{|B|} |\{ b \in B_f: \tau\sigma(f)\pi( b) \neq b\}| + \frac{\rho(\varepsilon/3)}{| B|}(|B_f|- |\{ b \in B_f: \tau\sigma(f)\pi( b) \neq b\}|)\\
&= \frac{\rho(\varepsilon/3)}{|B|}|B_f| + \frac{1-\rho(\varepsilon/3)}{|B|}|\{ b \in B_f: \tau\sigma(f)\pi( b) \neq b\}|\\
& \geq\frac{|B_f|}{|B|}\rho(\varepsilon/3)\geq (1-\varepsilon/8)\rho(\varepsilon/3)>(1-\varepsilon/3)\rho(\varepsilon/3).\qedhere
\end{align*}
\vspace{-0.3cm}\phantom\qedhere
\end{proof}
\textbf{Acknowledgments.} J. Brude was supported in part by a CONICET Doctoral Fellowship. R. Sasyk was supported in part through the grant PIP-CONICET 11220130100073CO. We thank Prof. Goulnara Arzhantseva for her comments and for clarifying a historical inaccuracy in the first draft of this article.
We thank the referee for his o her careful reading of the manuscript, in particular Section \ref{S5} was added thanks to his/her comments. 

%
%
%
%
%
%
\end{document}